\documentclass[11pt]{llncs}
\usepackage[letterpaper,hmargin=1.25in,vmargin=1.25in]{geometry}
\usepackage{graphics}
\usepackage{graphicx}
\usepackage[english]{babel}
\usepackage{amsmath}
\usepackage{amssymb}
\usepackage{amsfonts}
\usepackage{latexsym}
\usepackage{epsfig}
\usepackage{hyperref}

\usepackage[latin1]{inputenc}
\usepackage{times}
\usepackage[T1]{fontenc}

\usepackage{url}
\usepackage{lscape}

\newcommand{\MF}{\mathbb{F}}
\newcommand{\compare}{\stackrel{\triangle}{=}}

\newcommand{\etal}{{\it et~al.}}

\title{Computing the Number of Finite Field Elements\\ with Prescribed Trace and Co-trace}

\author{
        Assen Bojilov\inst{1} and
        Lyubomir Borissov\inst{1,2} and
        Yuri Borissov\inst{2}
       }

\institute{$^1$ Faculty of Mathematics and Informatics, Sofia University, J. Baucher Str. 2, Sofia, Bulgaria,\\
$^2$ Institute of Mathematics and Informatics, BAS, G. Bontchev Str. 8, 1113 Sofia,
Bulgaria.\\
  \vspace{1mm}
  \email{  bojilov@fmi.uni-sofia.bg, lborisov@fmi.uni-sofia.bg, youri@math.bas.bg}\\
  \vspace{0.5cm}
  {\it Dedicated to Prof. Tor Helleseth's 70th Birthday}
  }

\begin{document}

\maketitle
\pagestyle{plain}

\begin{abstract}\noindent
In this paper, we address the problem for determining the number of finite field elements with prescribed trace and co-trace in case of arbitrary characteristic $p$. We show that this problem can be reduced to solving a system of $p-1$ linear equations with matrix of coefficients a slight modification of circulant matrix formed by the Kloosterman sums over the field $\MF_{p}$. The presented approach is illustrated in the cases of characteristic $p = 2,3$ and $5$.
\end{abstract}

{\bf keywords:} trace function,  Kloosterman sum, circulant matrix, asymptotic behavior,\\ \hspace*{2.45cm}Andr\'e Weil bound.

\section{\bf Introduction}
\label{sect0}
The finite fields play an important role in coding theory, sequence design and cryptography.
One of the most useful tool for investigations in these scientific disciplines is the so-called trace function over finite field.
For a plenty of results involving this notion, we refer to \cite{Handbook}.

Some less-known applications of the trace are those considered in \cite{Dodu86} and \cite{Nied90}.
More specifically, S.~M. Dodunekov has proved the quasiperfectness of three classes of double-error correcting codes  using essentially the fact that in every binary field with degree of extension at least $3$ there exists a non-zero element whose trace is opposite to the trace of its inverse \cite{Dodu86}.
Also,  H. Niederreiter in his efforts  to establish a formula for the number of binary irreducible polynomials of given degree with second and next to the last coefficient both equal to $1$, has obtained as a by-product an expression for the cardinality of non-zero elements of a given binary field with property that their traces and the traces of inverses of them are equal to $1$  \cite{Nied90}.

Of course, the above mentioned two topics can be considered from a common perspective. So, the goal of this paper is to extend and generalize these earlier results in case of the finite fields with characteristic greater than $2$.

The paper is organized as follows. In the next section, we recall the background needed to present our results. Then in Section \ref{sect2}, we describe our approach for computing the cardinality of finite field elements with prescribed trace and co-trace in case of arbitrary characteristic. In Section \ref{sect3}, the asymptotic behavior of the quantities of interest is investigated. The examples for small values of characteristic are given in Section \ref{sect4}. Finally, some conclusions are drawn in the last section.

\section {\bf Preliminaries}
\label{sect1}
Let $\MF_{q}$ be the finite field of characteristic $p$ and order $q = p^{m}$, and let $\MF_{q}^{*}$ stand for the multiplicative group of $\MF_{q}$.
\begin{definition} The {\em trace} of an element  $\gamma$ \ in  $\MF_q$ over  $\MF_{p}$  is equal to
\begin {equation*}
tr(\gamma)= \gamma+ \gamma^{p}+ ... + \gamma^{p^{m-1}}
\end{equation*}
The {\em co-trace} of an element  $\gamma$ \ in  $\MF_q^{*}$  is equal to $tr(\gamma^{-1})$.
\end{definition}
It is well-known that the traces lie in the prime field $\MF_{p}$, and the number of elements in $\MF_{q}$ with fixed trace equals $q/p$ (see, e.g.  \cite[Ch. 4.8]{McWSl}).

\medskip

\noindent
For arbitrary $i,j \in \MF_{p}$ (with slight abuse in notations) we introduce the following:
\[
{ T_{ij} = |\{ x \in \MF_{q}^{*}:tr(x)=i, tr(x^{-1})=j)\}|},
\]
\noindent
i.e. $T_{ij}$ stands for the number of non-zero elements of $\MF_{q}$ with trace $i$ and co-trace $j$.

\medskip

\noindent
As it is prompted in the Introduction, we search for an approach to finding out closed-form formulae for {\rm $T_{ij}$} in terms of $m$ and $p$ when the characteristic $p$ is arbitrary.

\medskip

\noindent
We need as well the notion of {\em Kloosterman sums} for finite fields defined as follows
\begin {definition}\label{Klsum} (see, e.g. \cite{Car}) For each $u \in \MF_{q}^{*}$
\[
{\cal K}^{(m)} (u)= \sum_{x\in  \MF_{q}^{*}} \omega^{\ tr(x+ \frac{u}{x})},
\]
where $\omega=e^{\frac{2\pi i}p}$.
\end {definition}
\noindent

\medskip

\noindent
The crucial fact, we make use of, is inferred by the main result of  L. Carlitz' work from 1969. Namely, if $u \in \MF^{*}_{p}$,~ the Kloosterman sum
${\cal K}^{(m)}(u)$ is explicitly expressible in terms of the degree of field extension $m$, the characteristic $p$ and the Kloosterman sum ${\cal K}(u) \compare {\cal K}^{(1)}(u)$. More specifically, the following statement stated hereinafter as a proposition, is valid:
\begin{proposition}\label{Carlitz}
(see, e.g. \cite [Eq. 1.3]{Car}) For arbitrary $u \in \MF^{*}_{p}$, it holds:
\[
{\cal K}^{(m)} (u) = (-1)^{m-1} 2^{1-m} \sum_{2r \leq m}{m \choose 2r}{\cal K}^{m-2r}(u)\big({\cal K}^{2}(u) - 4p\big)^{r}
\]
\end{proposition}

\medskip

\noindent
It is deserved mentioning that there are other explicit expressions of this kind \cite{Zin} (in particular, by using the Dickson polynomials \cite{MoiRan}).
\medskip

\noindent
In addition, we exploit another three facts which can be found, for instance, in a slightly modified form in Lehmers' work from $1967$:
\begin{proposition}\label{Lehmer}
(see, \cite [Eqs.  1.9,  3.7 and 3.6, respectively]{Lehmer})
\[
{\rm (i)}\hspace{0.5cm} \sum_{u = 1}^{p-1} {\cal K} (u) = 1.
\]
\[
{\rm (ii)}\hspace{0.5cm} \sum_{u = 1}^{p-1} {\cal K}^{2}(u) = p^{2} - p - 1,
\]
and if $p > 2$
\[
{\rm (iii)}\hspace{0.5cm} \sum_{u = 1}^{p-1} {\cal K}(u) {\cal K}(cu) = -p - 1
\]
 for any $c \not = 1$ in $\MF^{*}_{p}$.
\end{proposition}

\medskip

\noindent
An another notion we need, is that of left-circulant matrix and for the sake of completeness we recall it as well as some basic properties of this kind of matrices (see, e.g. Carmona \etal \cite{Carmona}).
\begin{definition}
An $n \times n$ matrix ${\bf A}$ is called a {\em left-circulant matrix} if the $i-$th row of ${\bf A}$
is obtained from the first row of ${\bf A}$ by a left cyclic shift of $i - 1$ steps, i.e. the general
form of the left-circulant matrix is
\[
{\bf A} =
\left[ \begin{array}{cccccccc}
\;\;a_{0}\;\;a_{1}\;\;a_{2}\;\;...\;\;a_{n-2}\;\;a_{n-1}\\
a_{1}\;\;a_{2}\;\;a_{3}\;\;...\;\;a_{n-1}\;\;a_{0}\\
a_{2}\;\;a_{3}\;\;a_{4}\;\;...\;\;\;a_{0}\;\;\;\;\;a_{1}\\
.\;\;\;\; .\;\;\;\; .\;\;\; .\;\;\;\; .\;\;\; .\;\;\; .\;\;\; .\\
\;\;\;a_{n-1}a_{0}\;a_{1}\;\;...\;\;a_{n-3}\;\;a_{n-2}\\
\end{array} \right].
\]
\end{definition}

\medskip

\noindent
Apparently, the left-circulant matrices are symmetric and it is well-known that the inverse of a real invertible matrix of this type is again left-circulant.

\medskip

\noindent
The determinant of a left-circulant matrix can be computed using the following proposition.
\begin{proposition}\label{detcirc} \ %
Let ${\bf A}$ be a left-circulant matrix with first row $(a_0,a_1,\ldots,a_{n-1})$. Then:
\[
\det{\bf A} = (-1)^{(n-1)(n-2) \over 2}\prod_{l = 0}^{n-1} {f(\theta_{l})},
\]
where $f(x) = \sum_{r = 0}^{n-1} {a_{r} x^{r}}$ and $\theta_{l},\;l = 0,1,\ldots,n-1$ are the $n^\text{th}$ roots of unity.
\end{proposition}
In fact, the above formula differs from the "classical" formula for determinant of circulant matrix up to a sign.
The reason is that any left-circulant matrix can be obtained from the truly circulant matrix having the same first row by ${(n-1)(n-2) \over 2}$ transpositions of rows (from second to last one).

\medskip

\noindent
A less-known property of the left-circulant matrices is given by the next.
\begin{proposition}\label{sum'}
Let ${\bf A}$ be a left-circulant invertible matrix whose sum of the first row elements equals to  $S$.
Then for the sum $S^{\prime}$ of the first row elements of inverse matrix ${\bf A}^{-1}$ it holds $S^{\prime} = 1/S$.
\end{proposition}
\begin{proof}
Let the first rows of the matrices ${\bf A}$ and ${\bf A}^{-1}$  be $(a_0,a_1,\ldots,a_{n-1})$ and $(a^{\prime}_0,a^{\prime}_1,\ldots,a^{\prime}_{n-1})$, respectively. Using the defining matrix equality, i.e. ${\bf A} {\bf A}^{-1} = {\bf I}_{n},$ where ${\bf I}_{n}$ is the identity matrix of size $n$, we obtain the following (corresponding to the first column of the identity matrix) equalities:
\[
\sum_{l = 0}^{n-1}a_{l} a^{\prime}_{l} = 1
\]
\[
\sum_{l = 0}^{n-1}a_{s+l} a^{\prime}_{l} = 0,\; s = 1,\ldots,n-1,
\]
where the subscript $s+l$ is taken modulo $n$.
Summing up these $n$ equalities, after some rewriting, one gets: $\sum_{l=0}^{n-1}a_{l} \times \sum_{l=0}^{n-1}a^{\prime}_{l} = 1$ which completes the proof.
\qed
\end{proof}
\medskip

\noindent
We shall need as well the next.
\begin{lemma}\label{xy}
Let ${\bf A}_{n}$ be an $n \times n$ matrix having entries equal to $x$ over its main diagonal and equal to $y$ outside of the main diagonal. Then it holds:
\[
\Delta_{n} \compare \det {\bf A}_{n} = \big(x+(n-1)y\big)\big(x - y\big)^{n-1}.
\]
\end{lemma}
\begin{proof}
The following chain of computations is straightforward using basic properties of determinants (here the employed notation for determinant is $|{\bf .}|$ ):
\begin{multline*}
\Delta_{n}=
\begin{vmatrix}
x&y&y&\dots&y\\
y&x&y&\dots&y\\
y&y&x&\dots&y\\
\multispan{5}\strut\dotfill\\
y&y&y&\dots&x\\
\end{vmatrix}
=
\begin{vmatrix}
1&1&1&\dots&1\\
0&1&0&\dots&0\\
0&0&1&\dots&0\\
\multispan{5}\strut\dotfill\\
0&0&0&\dots&1\\
\end{vmatrix}
\begin{vmatrix}
x&y&y&\dots&y\\
y&x&y&\dots&y\\
y&y&x&\dots&y\\
\multispan{5}\strut\dotfill\\
y&y&y&\dots&x\\
\end{vmatrix}
=
\big(x+(n-1)y\big)
\begin{vmatrix}
1&1&1&\dots&1\\
y&x&y&\dots&y\\
y&y&x&\dots&y\\
\multispan{5}\strut\dotfill\\
y&y&y&\dots&x\\
\end{vmatrix}
=\\
\big(x+(n-1)y\big)
\begin{vmatrix}
1&0&0&\dots&0\\
-y&1&0&\dots&0\\
-y&0&1&\dots&0\\
\multispan{5}\strut\dotfill\\
-y&0&0&\dots&1\\
\end{vmatrix}
\begin{vmatrix}
1&1&1&\dots&1\\
y&x&y&\dots&y\\
y&y&x&\dots&y\\
\multispan{5}\strut\dotfill\\
y&y&y&\dots&x\\
\end{vmatrix}
=
\big(x+(n-1)y\big)
\begin{vmatrix}
1&1&1&\dots&1\\
0&x-y&0&\dots&0\\
0&0&x-y&\dots&0\\
\multispan{5}\strut\dotfill\\
0&0&0&\dots&x-y\\
\end{vmatrix}
=\\
\big(x+(n-1)y\big)\big(x - y\big)^{n-1}
\end{multline*}
\qed
\end{proof}
\medskip

\noindent
For the sake of completeness, we finally recall the famous Andr\'e Weil bound on Kloosterman sums over finite fields.
Namely, for any finite field $\MF_{q}$ of characteristic $p$ and every $u \in \MF^{*}_{q}$, the corresponding Kloosterman sum (see, Definition \ref{Klsum})
satisfies the Weil bound $|{\cal K}^{(m)} (u)| \leq 2\sqrt{p^{m}}$.

\section{\bf An outline of our approach}
\label{sect2}
The approach to achieving the goal of this paper consists of the following three basic steps:
\begin{itemize}
\item reducing the number of unknowns;
\item working out a system of linear equations;
\item the uniqueness of solution.
\end{itemize}
Hereinafter, we consecutively exhibit them.

\subsection{reducing the number of unknowns}
First, we shall prove the following proposition.
\begin{proposition}\label{prmain} \ %
For arbitrary $i,j \in \MF_{p}$, it holds:
\[
(\emph{i})\;\; T_{ij} = T_{ji},
\]
\noindent
and for $i \in \MF^{*}_{p}$:
\[
(\emph{ii})\;\; T_{ij} = T_{1 s},
\]
where $s = ij$.
\end{proposition}
\begin{proof}
\noindent
The obvious $(x^{-1})^{-1}=x$ for any $x \not = 0$ implies (\emph{i}).
Claim (\emph{ii}) follows by the fact that the mapping $x \to x/i$ is an $1-1$ correspondence on
$\MF_{q}$, and the next easily verifiable relations:
\[
tr(x/i) = tr(x)/i; \;\; tr((x/i)^{-1}) = tr(i\;x^{-1}) = i\;tr(x^{-1}),
\]
valid for any $i \in \MF^{*}_{p}$
\qed
\end{proof}
As an immediate corollary it is obtained.
\begin{corollary}\label{cor1.}
For any $i \in \MF^{*}_{p}$, it holds: $T_{0i} = T_{i0} = T_{10} = T_{01}$.
\end{corollary}
Proposition \ref{prmain} and Corollary \ref{cor1.} imply that it is sufficient to find closed-form formulae for $T_{00}, T_{01}$ and $T_{1s}, s = 1, \ldots, p-1$.

\medskip

\noindent
Moreover, one easily deduces the following:
\begin{lemma}\label{lemma1}
\begin{eqnarray}\label{eq.1}
T_{10} = q/p - \sum_{s = 1}^{p-1} T_{1s}
\end{eqnarray}
\begin{eqnarray}\label{eq.2}
T_{00} =(p - 1)\sum_{s = 1}^{p-1} T_{1s} + 2q/p - q - 1
\end{eqnarray}
\end{lemma}
\begin{proof}
Indeed, Eq. (\ref{eq.1}) is an immediate consequence of the fact that  the number of elements in $\MF_{q}$ with fixed trace equals $q/p$,
while Eq. (\ref{eq.2}) follows by the same fact (excluding the zero of the field) and taking into account Corollary \ref{cor1.} as well as just proven Eq. (\ref{eq.1}).
\qed
\end{proof}
Lemma \ref{lemma1} means that $T_{00}$ and $T_{01} = T_{10}$ can be expressed in terms of the quantities $T_{1s}, s = 1, \ldots, p-1$ and the characteristic $p$.

\subsection{working out a system of linear equations for $T_{1s}$}
Our aim in this subsection is to find a system of linear equations for the $p - 1$ unknowns $t_{s} \compare T_{1s}$.

\medskip

\noindent
To this end, for each $u \in \MF^{*}_{p}$ we proceed as follows:
\[
{\cal K}^{(m)}(u) \compare \sum_{x \in \MF_{q}^{*}} \omega^{tr(x + u x ^{-1})} = \sum_{i,j = 0}^{p-1} T_{ij} \omega^{i+uj} =
\]
\[
T_{00} + \sum_{j = 1}^{p-1} T_{0j}\omega^{uj} + \sum_{i = 1}^{p-1} T_{i0}\omega^{i} + \sum_{i,j = 1}^{p-1} t_{ij} \omega^{i+uj} =
\]
\[
T_{00} - 2 T_{01} + \sum_{s = 1}^{p-1} t_{s}(\sum_{i = 1}^{p-1}\omega^{i + {us \over i}}) = T_{00} - 2 T_{01} + \sum_{s = 1}^{p-1}{\cal K}(us)t_{s}.
\]
Notice that in the above chain of elementary transformations we have consecutively used Proposition \ref{prmain}.(\emph{ii}), Corollary \ref{cor1.} \; and the fact that $\omega = e^{\frac{2 \pi i}p}$ is a nontrivial $p-$th root of unity.

\medskip

\noindent
Further on, making use of Lemma \ref{lemma1} in the latter equality and after a minor rewriting, we get the following linear system:
\begin{eqnarray}\label{eq.3}
\sum_{s = 1}^{p-1}\big({\cal K}(us)+p+1\big)t_{s} = {\cal K}^{(m)}(u) + q + 1,\; u \in \MF_{p}^{*}.
\end{eqnarray}
Note that, for fixed $u$, the right-hand side of (\ref{eq.3}) can be expressed in terms of ${\cal K}(u), m$ and $p$ taking into consideration Carlitz' result (Proposition \ref{Carlitz}).

\subsection{the uniqueness of solution}

\noindent
\begin{proposition}\label{propequiv}
The linear system (\ref{eq.3}) is equivalent to a system whose matrix of coefficients is left-circulant.
\end{proposition}
\begin{proof}
Let $g$ be a primitive element of $\MF^{*}_{p}$. Renaming the unknowns by $x_{l} \compare t_{g^{l}}, l = 0, \ldots, p-2$ and arranging equations (\ref{eq.3}) according to increasing order of powers of $g$ in their right-hand sides, one gets a system of the form:
\begin{eqnarray}\label{eq.4}
\sum_{l = 0}^{p-2}k_{s+l} x_{l} = {\cal K}^{(m)}(g^s) + q + 1,\;\; s = 0, \ldots, p-2,
\end{eqnarray}
where the subscript of $k_{s+l} \compare {\cal K}(g^{s+l}) + p + 1$ is taken modulo $p-1$, of course.
Obviously, the matrix ${\bf K} \compare {\bf K}(g)$ of coefficients of system (\ref{eq.4}) is a real left-circulant matrix.
\qed
\end{proof}

\medskip

\noindent
Now, we prove the following.
\begin{lemma}\label{lem2}
\[
\det {\bf K} = p^{2} \det {\bf K}^{\prime},
\]
where ${\bf K}^{\prime}$ is the left-circulant matrix having as a first row the vector $\big({\cal K}(1),{\cal K}(g),{\cal K}(g^{2}),\ldots,{\cal K}(g^{p-2})\big)$.
\end{lemma}
\begin{proof}

\medskip

\noindent
We make use of Proposition \ref{detcirc}. There are two essentially distinct cases to be considered:

1)\;if $\theta = 1$, we easily derive:
\[
\sum_{l=0}^{p-2}k_{l}\theta^{l} = \sum_{l=0}^{p-2}\big({\cal K}(g^{l}) + p + 1\big) = \sum_{l=0}^{p-2}{\cal K}(g^{l}) + p^{2} - 1 =
\]
\[
p^{2} = p^{2} \times 1 = p^{2} \sum_{l=0}^{p-2}{\cal K}(g^{l})\theta^{l}
\]
by using Proposition \ref{Lehmer}({\rm i}),

2)\; otherwise, we have:
\[
\sum_{l=0}^{p-2}k_{l} \theta^{l} = \sum_{l=0}^{p-2}\big({\cal K}(g^{l})\theta^{l} + (p + 1)\theta^{l}\big) = \sum_{l=0}^{p-2}{\cal K}(g^{l})\theta^{l},
\]
since $\theta$ is a nontrivial $(p-1)^\text{st}$ root of unity.
\qed
\end{proof}

\medskip

\noindent
\begin{proposition}\label{prop3}
\[
|\det {\bf K}^{\prime}| = p^{p-2}
\]
\end{proposition}
\begin{proof}
If $p = 2$ the claim is trivial.
In case $p > 2$, using Proposition \ref{Lehmer}({\em ii})--({\em iii}), one shows that the matrix ${{\bf K}^{\prime}}^{2}$ satisfies the assumptions of Lemma \ref{xy} with $x = p^{2} - p - 1$ and $y = -p - 1$. Thus, $\det^{2} {\bf K}^{\prime} = p^{2(p-2)}$ which completes the proof.
\qed
\end{proof}
\begin{remark}
Note that the determinants of the matrices under consideration do not depend on particular chosen primitive element $g$ but only on the field characteristic.
\end{remark}

\medskip

\noindent
Now, we are in position to establish the main result.
\begin{theorem}\label{th1}
The quantities $T_{1s}, s = 1,\ldots,p-1$ can be found as the unique solution of linear system (\ref{eq.3}).
\end{theorem}
\begin{proof}
Indeed, Lemma \ref{lem2} and Proposition \ref{prop3} immediately imply $|\det {\bf K}| = p^{p}$, i.e. the matrix  ${\bf K}$ of coefficients of equivalent system (\ref{eq.4}) (see, Proposition \ref{propequiv}) is invertible.
\qed
\end{proof}
\begin{remark}
Of course, the remaining $T_{ij}$ can be found by applying Lemma \ref{lemma1} and Proposition \ref{prmain}.
\end{remark}

\section{\bf Asymptotic behavior of the quantities $T_{ij}$}
\label{sect3}

In this section, at first, we will give a convenient matrix presentation of the unique solution of system (\ref{eq.4}). To this end, we introduce some additional notations, i.e. denote by ${\bf q} = (q+1, q+1,\ldots,q+1)$ and by ${\bf k}^{(m)} = \big({\cal K}^{(m)}(1), {\cal K}^{(m)}(g),\ldots,{\cal K}^{(m)}(g^{p-2})\big)$, the vectors of length $p-1$ with constant coordinate $q+1$ and of the Kloosterman sums over $\MF_q$, respectively.
Apparently, the system of interest has the following matrix form:
\[
{\bf x}{\bf K} =  {\bf q} + {\bf k}^{(m)},
\]
with ${\bf x} = (x_0,x_1,\ldots,x_{p-2})$ being the vector of unknowns.

\medskip

\noindent
If ${\bf R}$ is the inverse of matrix ${\bf K}$ (see, Theorem \ref{th1}) then the above matrix equation can be solved as follows:
\begin{equation}\label{eq.5}
{\bf x} =  {\bf q}{\bf R} + {\bf k}^{(m)}{\bf R}.
\end{equation}
\begin{lemma}
In the above notations, the product ${\bf q}{\bf R}$ equals to the vector $(\frac{q+1}{p^{2}},\frac{q+1}{p^{2}},\ldots,\frac{q+1}{p^{2}})$.
\end{lemma}
\begin{proof}
Since ${\bf R}$ is a left-circulant matrix and ${\bf q}$ is with constant coordinates, it can be easily seen that ${\bf q}{\bf R}$ has components equal to the constant $ (q+1) \times \sum_{l = 0}^{p-2}r_{l}$ where $r_0,r_1,\ldots,r_{p-2}$ are the elements of first column (row) of ${\bf R}$. But according to Proposition \ref{sum'}, it holds $\sum_{l = 0}^{p-2}r_{l} = 1/{\sum_{l=0}^{p-2}k_{l}}$. Finally, the proof is completed observing that the last sum $\sum_{l=0}^{p-2}k_{l}$ equals $p^{2}$ (see, part 1 of the proof of Lemma \ref{lem2}).
\qed
\end{proof}

\medskip

\noindent
The above proposition and Eq. (\ref{eq.5}) mean that each unknown can be found of the form:
\begin{equation}\label{eq.6}
x_{s} = \frac{q}{p^{2}}+\frac{1}{p^{2}} + \sum_{l=0}^{p-2}r_{s+l}{\cal K}^{(m)}(g^{l}),\; s = 0,\ldots,p-2,
\end{equation}
where the index $s+l$ is taken modulo $p-1$.

\begin{theorem}\label{th2}
$\lim_{m \to \infty}\frac{T_{1s}}{p^{m}} = {1 \over {p^2}}$ for each $1 \leq s \leq p-1$.
\end{theorem}
\begin{proof}
We make use of Eq. (\ref{eq.6}). By the Weil bound on Kloosterman sums it follows: $|{\cal K}^{(m)}(u)|/p^{m} \leq 2/\sqrt{p^m},$ for each $u \in \MF_{p}$. This together with the fact that $r_0,r_1,\ldots,r_{p-2}$ are constants (depending only on characteristic $p$) implies that the behavior of ratio $\frac{T_{1s}}{p^{m}}$ when $m$ is sufficiently large is determined by the term $q/p^{m+2}$, i.e. it resembles $1/p^{2}$.
\qed
\end{proof}
\begin{corollary}
Both ratios $\frac{T_{00}}{p^{m}}$ and $\frac{T_{01}}{p^{m}}$ converge to $1/p^{2}$ when  $m$ goes to infinity.
\end{corollary}
\begin{proof}
An immediate consequence of Lemma \ref{lemma1} and Theorem \ref{th2}.
\qed
\end{proof}

\section{\bf Examples}
\label{sect4}

\subsection{{\cal char} = 2}

\medskip

\noindent
Combining the solely (in this case) Eq. (\ref{eq.3}) and Carlitz' result, we get:
\begin{eqnarray}\label{eq.7}
T_{11} =  \frac{2^m + 1} {4} + \frac{1} {2^{m+1}} \sum_{r = 0}^{\lfloor m/2 \rfloor}(-1)^{m+r+1}{m \choose 2r} 7^{r}.
\end{eqnarray}
Also, by using Lemma \ref{lemma1} we obtain:
\[
T_{00} =  \frac{2^m - 3} {4} + \frac{1} {2^{m+1}} \sum_{r = 0}^{\lfloor m/2 \rfloor}(-1)^{m+r+1}{m \choose 2r} 7^{r}
\]
and
\[
T_{01} = \frac{2^m - 1} {2} - \frac{1} {2^m} \sum_{r = 0}^{\lfloor m/2 \rfloor}(-1)^{m+r+1}{m \choose 2r} 7^{r},
\]
respectively  \cite{Bor16}.
\begin{remark}
Formula (\ref{eq.7}) is obtained as a by-product in \cite{Nied90} without making use of Carlitz' result.
\end{remark}
{\bf Table 1.} in the Appendix shows some numerical results yielded by using the above formulae.

\subsection{{\cal char} = 3}

\medskip

\noindent
Solving system (\ref{eq.3}), we get:
\begin{align*}
T_{11} =  \frac{3^m + 1} {9} + \frac{ - {\cal K}^{(m)}(1) + 2{\cal K}^{(m)}(2)} {9},\\
T_{12}  =  \frac{3^m + 1} {9} + \frac{2{\cal K}^{(m)}(1) - {\cal K}^{(m)}(2)} {9}\\
\end{align*}
and by using Lemma \ref{lemma1}
\begin{align*}
T_{00} = \frac{3^m - 5} {9} +  \frac{2{\cal K}^{(m)}(1) + 2{\cal K}^{(m)}(2)}{9 },\\
T_{01} = \frac{3^m - 2} {9} - \frac{{\cal K}^{(m)}(1) + {\cal K}^{(m)}(2)}{9}.\\
\end{align*}
Now, Carlitz' result can be applied but we skip the final formulae due to their too cumbersome form.

\medskip

\noindent
{\bf Table 2.} in the Appendix presents some numerical results for this case.

\subsection{{\cal char} = 5}
The solution of system (\ref{eq.3}) in this case is:
\begin{align*}
T_{11}=&\frac{5^m+1}{25}{\small+}\frac{2{\cal K}^{(m)}(1)-2{\cal K}^{(m)}(3)+{\cal K}^{(m)}(4){\small+}
\phi{\small\times}\big({\cal K}^{(m)}(1)+2{\cal K}^{(m)}(2)-2{\cal K}^{(m)}(3)-{\cal K}^{(m)}(4)\big)}{25},\\
T_{12}=&\frac{5^m+1}{25}{\small+}\frac{{\cal K}^{(m)}(2)+2{\cal K}^{(m)}(3)-2{\cal K}^{(m)}(4){\small+}
\phi{\small\times}\big(2{\cal K}^{(m)}(1)-{\cal K}^{(m)}(2)+{\cal K}^{(m)}(3)-2{\cal K}^{(m)}(4)\big)}{25},\\
T_{13}=&\frac{5^m+1}{25}{\small+}\frac{-2{\cal K}^{(m)}(1)+2{\cal K}^{(m)}(2)+{\cal K}^{(m)}(3){\small+}
\phi{\small\times}\big(-2{\cal K}^{(m)}(1)+{\cal K}^{(m)}(2)-{\cal K}^{(m)}(3)+2{\cal K}^{(m)}(4)\big)}{25},\\
T_{14}=&\frac{5^m+1}{25}{\small+}\frac{{\cal K}^{(m)}(1)-2{\cal K}^{(m)}(2)+2{\cal K}^{(m)}(4){\small+}
\phi{\small\times}\big(-{\cal K}^{(m)}(1)-2{\cal K}^{(m)}(2)+2{\cal K}^{(m)}(3)+{\cal K}^{(m)}(4)\big)}{25},\\
\end{align*}
where $\phi = \frac{-1+\sqrt5}2$.\\
Finally, Lemma \ref{lemma1} gives:
\begin{align*}
T_{00}=&\frac{5^m-9}{25} + \frac{{\cal K}^{(m)}(1)+{\cal K}^{(m)}(2)+{\cal K}^{(m)}(3)+{\cal K}^{(m)}(4)}{25},\\
T_{01}=&\frac{5^m-4}{25} -\frac{{\cal K}^{(m)}(1)+{\cal K}^{(m)}(2)+{\cal K}^{(m)}(3)+{\cal K}^{(m)}(4)}{25}.\\
\end{align*}
\noindent
{\bf Table 3.} in the Appendix contains some numerical results for this case.

\section{\bf Conclusion}
\label{sect5}
In this paper, we describe an approach to computing the number of elements of the finite field $\MF_{q},\;q = p^{m}$
with prescribed trace and co-trace. This approach consists of reducing the problem under consideration to solving a linear system with
coefficient matrix which is a slight modification of circulant matrix formed by the Kloosterman sums over the prime
field $\MF_{p}$. Also, for that system we prove the uniqueness of its solution based on some deep properties of the
sums considered. Together with 1969's result of Carlitz giving explicit formulas for corresponding Kloosterman sums
over $\MF_{q}$ in terms of $m$ and $p$, this allows to find closed-form expressions for the quantities of interest.
The study of asymptotic behavior of these quantities shows that they resemble $p^{m-2}$ when $m$ is sufficiently large.
The last fact can be interpreted as the constraints imposed on the values of trace and co-trace are in some sense
independent to each other.

\section*{\it \bf Acknowledgments}
The research of the first author was supported, in part, by the Science Foundation of Sofia University under contract 80-10-74/20.04.2017. The research of the second and third authors was partially supported by the foundation "Georgi Tchilikov".

\bibliography{}

\newpage

\begin{appendix}
\begin{center}
\textsc{Appendix}
\end{center}
{\Large
\begin{table}\label{Table1}
\caption {Values of $T_{ij}$ for $2 \leq m \leq 10$, {\cal char} = $2$.}
\begin{center}
\begin{tabular}{|l|l|l|l|l|l|l|l|l|l|}
\hline
$m$ &   $2$   & $3$  &  $4$  &  $5$  &  $6$  &  $7$  &  $8$ & $9$ & $10$\\
\hline
$T_{00}$ &   $1$   & $0$ &  $3$ & $10$ &  $13$ &  $28$&  $71$ & $126$ & $241$\\
\hline
$T_{01}$  &  $0$   & $3$ &  $4$ &  $5$ &  $18$ &  $35$&  $56$ & $129$ & $270$\\
\hline
$T_{11}$  &  $2$   & $1$ &  $4$ & $11$ &  $14$ &  $29$&  $72$ & $127$ & $242$\\
\hline
\end{tabular}
\end{center}
\end{table}
\begin{table}\label{Table2}
\begin{center}
\caption {Values of $T_{ij}$ for $1 \leq m \leq 6$, {\cal char} = $3$.}
\begin{tabular}{|l|l|l|l|l|l|l|}
\hline
$m$ &   $1$   & $2$  &  $3$  &  $4$  &  $5$  &  $6$\\
\hline
$T_{00}$ &   $0$   & $2$ &  $2$ & $10$ &  $20$ &  $68$\\
\hline
$T_{01}$  &  $0$   & $0$ &  $3$ &  $8$ &  $30$ &  $87$\\
\hline
$T_{11}$  &  $1$   & $1$ &  $0$ & $13$ &  $31$ &  $72$\\
\hline
$T_{12}$  &  $0$   & $2$ &  $6$ & $6$ &  $20$ &  $84$\\
\hline
\end{tabular}
\end{center}
\end{table}
\begin{table}\label{Table3}
\begin{center}
\caption {Values of $T_{ij}$ for $1 \leq m \leq 6$, {\cal char} = $5$.}
\begin{tabular}{|l|l|l|l|l|l|l|}
\hline
$m$ &   $1$   & $2$  &  $3$  &  $4$  &  $5$  &  $6$\\
\hline
$T_{00}$ &   $0$   & $4$ &  $0$ & $28$ &  $164$ &  $628$\\
\hline
$T_{01}$  &  $0$   & $0$ &  $6$ &  $24$ &  $115$ &  $624$\\
\hline
$T_{11}$  &  $1$   & $2$ &  $0$ & $21$ &  $120$ &  $601$\\
\hline
$T_{12}$  &  $0$   & $0$ &  $6$ & $38$ &  $130$ &  $590$\\
\hline
$T_{13}$  &  $0$   & $2$ &  $6$ & $16$ &  $140$ &  $660$\\
\hline
$T_{14}$  &  $0$   & $1$ &  $7$ & $26$ &  $120$ &  $650$\\
\hline
\end{tabular}
\end{center}
\end{table}
}

\end{appendix}

\end{document}